\documentclass[10pt]{amsart}
\usepackage{amsfonts,amsmath,amssymb,amsthm,latexsym,verbatim}
\newtheorem{theorem}{Theorem}[section]
\newtheorem{lemma}[theorem]{Lemma}
\newtheorem{proposition}[theorem]{Proposition}

\theoremstyle{remark}
\newtheorem{remark}[theorem]{Remark}
\theoremstyle{definition}

\newtheorem{example}[theorem]{Example}
\newtheorem{problem}[theorem]{Problem}

\DeclareMathOperator{\trace}{{\mathrm{trace}}}
\DeclareMathOperator{\image}{{\mathrm{Im}}}
\DeclareMathOperator{\soc}{{\mathrm{soc}}}

\DeclareMathOperator{\rad}{{\mathrm{rad}}}

\DeclareMathOperator{\Hom}{{\mathrm{Hom}}}

\newcommand{\abs}[1]{|#1|}	
\newcommand{\F}{{\mathbf {F}}} 
 
\newcommand{\Z}{{\mathbf {Z}}} 
\newcommand{\Q}{{\mathbf {Q}}}

\newcommand{\HH}{{\mathcal {H}}} 
\newcommand{\LL}{{\mathcal {L}}}
\newcommand{\Tr}{{\mathrm {Tr}}} 
\newcommand{\Abar}{{\overline{{A}}}} 

\newcommand{\leftsub}[2]{{\vphantom{#2}}_{#1}{#2}}
\DeclareMathOperator{\GL}{{\mathrm {GL}}} 	

\DeclareMathOperator{\diag}{{\mathrm {diag}}} 	

\newcommand{\allone}{{\mathbf 1}} 	

\begin{document}
\title[Smith Normal Form]{Smith Normal Forms of Incidence Matrices}
\author{Peter Sin}
\address{Department of Mathematics\\University of Florida\\ P. O. Box 118105\\ Gainesville
FL 32611\\ USA}
\date{}
\thanks{This work was partially supported by a grant from the Simons Foundation (\#204181 to Peter Sin)}



\begin{abstract}
A brief introduction is given to the topic of Smith normal forms of incidence matrices. 
A general discussion of techniques is illustrated by some classical examples. 
Some recent advances are described and the limits of our current understanding  are
indicated.
\end{abstract}

\maketitle

\section{Introduction} An incidence matrix is a matrix $A$ of zeros and ones
which encodes a relation between two finite sets $X$ and $Y$.
Related elements are said to be {\it incident}. The rows of the incidence matrix $A$ are indexed by the elements of $X$, ordered in some way, and the
columns are indexed by $Y$. The $(x,y)$ entry is $1$ if $x$ and $y$ are incident 
and zero otherwise. Many of the incidence relations we shall consider will
be special cases or variations of the following basic examples.
\begin{example}\label{exsets}
Let $S$ be a set of size $n$ and for two fixed numbers
$k$ and $j$ with  $0\leq k\leq j\leq n$ let $X$ be the set of all subsets of $S$ of size $k$ and $Y$ the set of subsets of $S$ of size $j$. The most obvious incidence relation
is set inclusion but, more generally, for each $t\leq k$
we have a natural incidence relation in which elements of $X$ and $Y$ to be incident if their intersection has size $t$.
\end{example}

\begin{example}\label{exsubspaces}
The  example above has a ``$q$-analog'' in which $X$ and $Y$ are
the sets of $k$-dimensional and  $j$-dimensional subspaces of an $n$-dimensional vector space over a finite field $\F_q$, with incidence being inclusion, or more generally,
specified by the dimensions of the subspace intersections. It is common in this context
to use the terminology of projective geometry, referring to $1$-dimensional
subspaces as points, $2$-dimensional subspaces as lines, etc. of the projective space
$PG(n-1,q)$.
\end{example}

Further examples of incidence relations abound from graph theory and design theory,
and we will discuss both general classes and specific examples. 

We may view $A$ as having entries over any commutative ring with $1$ but
in this paper we shall always assume that the entries are integers, which is
the most general case, in the sense that many results for other rings such as fields can 
be deduced from results over $\Z$. 

An incidence matrix translates an incidence relation, with no loss of information, into
linear algebra. Thus, we are led inescapably to the study of its algebraic invariants.
In the case where $X=Y$, we could consider the spectrum of the square matrix $A$, or its rational canonical form.
For the general case, where $A$ is not necessarily square, the fundamental invariant
is the {\it Smith normal form of $A$}, whose definition we now recall.
A square integer matrix is {\it unimodular} if it is invertible in the ring
of integer matrices, which is the same as saying that its determinant is $\pm 1$.
Two integer matrices $A$ and $B$ are {\it equivalent over $\Z$} if there
exist unimodular matrices $P$ and $Q$ such that $B=PAQ$. It is a standard theorem
that every integer matrix is equivalent to one of the form
\begin{equation}\label{PAQ}
PAQ=\begin{pmatrix}
 s_1&0&0&0&\hdots\\
 0&\ddots&0&0&\hdots\\
 0&0&s_r&0&\hdots\\
 0&0&0&0&\hdots\\
 \vdots&\vdots&\vdots&\vdots&\ddots
\end{pmatrix},
\end{equation}
where $r$ is the rank of $A$ and  $s_i|s_{i+1}$ for $i=1$,\dots,$r-1$.
The entries $s_i$ are uniquely determined up to signs and the matrix
is called the Smith normal form  (SNF) of $A$.
(The form is named after H. J. S. Smith, who also showed \cite{Smith2} that $s_i=d_i/d_{i-1}$, where $d_0=1$ and $d_i$ for $i\geq 1$ is the greatest common divisor of the ($i\times i$) minors of $A$). 

The SNF is a natural choice of invariant for an incidence relation
as it does not depend on an arbitrary ordering of the sets. In other words, it
is the same for all possible incidence matrices of the relation. 

Of course, it is well known how to bring an integer matrix to Smith normal form by
applying row and column operations corresponding to left and right multiplications
by unimodular matrices in a systematic way, based on the euclidean algorithm.
However, we are interested in solving this problem not for one incidence
matrix at a time but for parametrized families of such matrices. 
The object is to describe the SNFs (or equivalent forms) 
uniformly as functions of the parameters.
Such computations could be expected to provide insight into the 
mathematical structure of the incidence relations.

The purpose of this article is to provide an introduction to the topic
of computing the SNFs of families of incidence matrices.
Thanks to the existence of a readable and thorough survey \cite{Xiang2} 
on the SNF of incidence matrices of designs, and the related question of their $p$-ranks, 
much of the history and literature has been covered, and we can concentrate instead on elucidating some of the algebraic techniques that have been introduced, and on describing some recent results and open questions. 

\section{Generalities on Smith normal form} 
Let $R$ be a principal ideal domain and $A$ an $m\times n$  matrix with entries in $R$.
In our examples, $R$ will be either $\Z$ or the localization at a nonzero prime
of the ring of integers in a number field. The definition of SNF can,
{\it mutatis mutandis}, be extended to $R$.
If we view $A$ as the matrix of the homomorphism $\eta: R^m\to R^n$ by matrix multiplication on the right, then the SNF is one of several ways of describing the cyclic decomposition of the {\it Smith module} $G(\eta)=R^n/\image(\eta)$, namely in its {\it invariant factor form}
\footnote{Some minor differences in terminology are found in the literature.
Many authors do not allow $1$ as an invariant factor or
elementary divisor, which is quite reasonable from the module-theoretic viewpoint.
However, in studying  matrices or maps it is often more convenient to think of the invariant factors as the list of the $r$ nonzero entries $s_i$ of the SNF, including those
equal to 1. Then, for $1\leq i\leq r$, the exact power of a prime $\pi$ dividing $s_i$ is called the $i$-th $\pi$-elementary divisor. This is the convention we shall adopt.}.

While the term ``SNF'' is traditional and is a useful 
label, it is really the Smith module $G(\eta)$ which is at the center of interest.  
A matrix equivalent to $A$ which has nonzero entries only on the
diagonal is called a {\it diagonal form} of $A$. 
Since a diagonal form  describes 
$G(\eta)$ up to isomorphism, it  also counts as a valid solution 
to the ``SNF problem''. 
Another alternative description of $G(\eta)$ is the {\it elementary divisor form},
which can be thought of as the set of $\mathfrak p$-local SNFs for all primes $\mathfrak p$ dividing the order of torsion subgroup of $G(\eta)$.

Different problems lend themselves to different decompositions of $G(\eta)$, so it is best to be flexible about the exact formulation of results.
For instance, in the case of inclusion of $k$-subsets 
in $j$-subsets of an  $n$-set, the most natural diagonal form turns out to have 
binomial coefficients as entries, with multiplicities equal to differences of binomial
coefficients. In this case, to give elementary divisors
would involve consideration of the exact power to which each prime divides
a binomial coefficient, which, although it is well known, would make the statements
much more complicated. It would be even more  challenging to give a uniform description 
of the invariant factor form, but the point is that the extra difficulty comes from the arithmetic of binomial coefficients and not from the incidence relation.
In the case of inclusion of points in linear subspaces of a projective space, 
it will be seen that if $p$ is the underlying characteristic, the most natural decomposition of the torsion subgroup of $G(\eta)$ is as the direct sum of a $p$-group, described by a $p$-local SNF, and a cyclic group of order prime to $p$.

Suppose $\mathcal B$ is a basis of $R^m$ and $\mathcal C$ is a basis
of $R^n$ such that the matrix of $\eta$ in these bases is
diagonal. Then we shall call $\mathcal B$  a {\it left SNF basis}
and $\mathcal C$ a {\it right SNF basis.} In terms of a matrix equation
$$
PAQ=D,
$$
with $P$ and $Q$ unimodular  (i.e. invertible over $R$) and $D$ in diagonal form, 
$\mathcal B$ corresponds to the rows of $P$ and $\mathcal C$ to the rows of
$Q^{-1}$.

\subsection{Local SNFs} \label{localsnf}
On occasion, we shall be forced to consider certain extensions of the PID
$R$. For example we may wish to adjoin roots of unity to $\Z$. 
In general this may take us out of the realm of PIDs into Dedekind domains, 
but since the SNF problem can be
solved one prime at a time we can localize, bringing us back to PIDs,
in fact to discrete valuation rings.
Therefore, we shall consider an extension $R\subset R'$ of PIDs and
compare $G(\eta)$, the cokernel of $\eta:M\to N$ with the cokernel
$G(1\otimes\eta)$ of the induced map $1\otimes\eta: R'\otimes_RM\to R'\otimes_RN$.
Since tensoring $R$-modules with $R'$ is a right exact functor,
we have
\begin{equation}
G(1\otimes_R\eta)\cong R'\otimes_RG(\eta).
\end{equation}
In the simplest situation, where
the prime $p\in R$ is unramified in $R'$, and $\pi$ is a prime of $R'$
above $p$, the multiplicity of $p^i$ as an elementary divisor of $\eta$
is equal to the multiplicity of $\pi^i$ as an elementary divisor of $1\otimes_R\eta$.   

Let $R$ be a discrete valuation ring of characteristic zero 
with fraction field $K$, maximal ideal $(\pi)$ and residue
field $k=R/(\pi)$ of characteristic $p>0$.
Given a homomorphism 
\begin{equation}\label{eta}
\eta:M\to N
\end{equation}
of free $R$-modules of finite rank, we define
\begin{equation}
M_i=M_i(\eta)=\{ m\in M \mid \eta(m)\in \pi^iN\},
\end{equation}
and
\begin{equation}
N_i=N_i(\eta)=\{ \pi^{-i}\eta(m) \mid m\in M_i\}.
\end{equation}
We have
\begin{equation}
M=M_0\supseteq M_1\supseteq M_2\supseteq \cdots
\end{equation}
and
\begin{equation}
\image \eta= N_0\subseteq N_1\subseteq N_2\subseteq \cdots
\end{equation}
Let $\overline M_i$ and $\overline N_i$ be the images of $M_i$ and $N_i$
in $\overline M=M/\pi M$ and $\overline N=N/\pi N$ respectively.

There exist $\ell$, $\ell'$ such that $\overline M_i=\overline{\ker\eta}$ for  $i\geq\ell$
and $N_j$ is equal to the purification of $\image\eta$ for $j\geq\ell'$.

Let $e_i(\eta)$ denote the multiplicity of $\pi^i$ as an elementary divisor
of $\eta$. Then, as is easily seen from considering the SNF of $\eta$, we have
\begin{equation}\label{eldiv}
e_i(\eta)=\dim(\overline M_i/\overline M_{i+1})=\dim(\overline N_{i}/\overline N_{i-1}).
\end{equation}
(We set $N_{-1}=0$.)
Left and right SNF bases can be constructed from the modules $M_i$ and $N_j$ as follows.
Choose bases $\overline B_i$ for the $\overline M_i$ that are ``nested'',
by which is meant that $\overline B_\ell\subseteq\overline B_{\ell-1}\subseteq\cdots\subseteq \overline B_0$. First we let $\overline D_\ell=\overline B_\ell$ and for $i=0$,\dots,$\ell-1$  we
set $\overline D_i=\overline B_i\setminus \overline B_{i+1}$. 
Let $D_\ell$ be a basis of $\ker\eta$ which maps onto $\overline D_\ell$
and for $i=0$,\dots, $\ell-1$, let $D_i\subseteq M_i$ be a set which maps bijectively onto $\overline D_i$. 
Then it is easily verified that $B=\bigcup_{j=0}^\ell D_j$ is a left SNF basis
for $\eta$. Similarly, a right SNF basis can be constructed by lifting
nested bases of the $\overline N_j$.
  
As observed in \cite{BDS} these notions are of importance when we wish to compare the 
SNF of a product $AB$ of two matrices with those of $A$ and $B$.  
Of course, it is easy to see that $s_i(A)$ and $s_i(B)$
divide $s_i(AB)$ but little more can be said in general 
except in trivial cases such as when the Smith modules are finite groups with coprime
orders.
As a simple example, let 
\begin{equation}
A=\begin{pmatrix}1&0\\
-p&p^2
\end{pmatrix}
\qquad\text{and}\qquad
B=\begin{pmatrix}p&0\\
1&p
\end{pmatrix}.
\end{equation}
Then $A$ and $B$ each have elementary divisors $1$ and $p^2$, while $AB$
has elementary divisors $p$ and $p^3$. 

Suppose that $n$ is the number of columns of $A$ and the number of rows of $B$,
and suppose that $R^n$ has a basis which is {\it simultaneously} a right SNF basis
for $A$ and a left SNF basis for $B$.  Then we have unimodular matrices $P$, $Q$ and $S$
such that
$$
PAQ=D,\quad Q^{-1}BS=D',
$$
with $D$ and $D'$ in diagonal form.

Hence $P(AB)S=DD'$ and we have  multiplicativity of the diagonal forms.

The following result (based on an argument in \cite{BDS}) shows how the
rare phenomenon of  multiplicativity of diagonal forms 
may arise from  certain  group actions.

\begin{proposition}\label{multfree} Let $K$ be the fraction field of $R$.  Suppose that
we have an abelian group $G$ of order prime to $p$, $R$-free $RG$-modules $X$, 
$Y$ and $Z$, and $RG$-module homomorphisms  
$\alpha:X\to Y$ and $\beta:Y\to Z$. Assume further that the action of $KG$ on
$K\otimes_RY$ is multiplicity-free, i.e. no two simple composition factors
are isomorphic. Then  $Y$ has a basis which is simultaneously
a right SNF basis for $\alpha$ and a left SNF basis for $\beta$.
\end{proposition}
\begin{proof} 
Let $\xi$ be  a primitive $\abs{G}$-th root of unity in an algebraic closure of $K$.
Then since $(\pi)$ is unramified in $R[\xi]$, the $\pi$-elementary divisors of 
the induced maps $1\otimes\alpha$, $1\otimes\beta$ and $1\otimes\beta\circ\alpha$
are the same as those for $\alpha$, $\beta$ and $\beta\circ\alpha$.
Also the multiplicity-free hypothesis is still valid if we extend the ring to $R[\xi]$.
Therefore, we may assume that $R=R[\xi]$. In this case, we have a decomposition
of $K\otimes_RY$ as a direct sum of one-dimensional components $K_\chi$,
where each $\chi$ is a character of $G$. The element
$$
h_\chi=\frac{1}{\abs{G}}\sum_{g\in G}\chi(g^{-1})g
$$
is a projection onto $K_\chi$ which lies in the center of $RG$, 
and $\sum_{\chi\in\Hom(G,R^\times)}h_\chi=1$.
Therefore, we also have the decomposition 
$Y=\oplus_\chi R_\chi$, where $R_\chi=Y\cap K_\chi$.
Let $v_\chi$ be a generator of $R_\chi$, so that the $v_\chi$ form a basis $\mathcal B$.
We claim that $\mathcal B$ is simultaneously a left SNF basis for $\beta$ and a right
SNF basis for $\alpha$.
The images $\overline v_\chi$ of the $v_\chi$ form a basis $\overline{\mathcal B}$ 
of $\overline Y$. The multiplicity-free condition means that for every submodule of $\overline Y$ there is a subset of $\overline{\mathcal B}$ which is a basis.
Thus, in the construction of a left SNF basis $B$ described in \S\ref{localsnf} (with $\eta=\beta$, $M=Y$ and  $N=Z$) we can take the nested bases $\overline B_i$ to 
be subsets of $\overline{\mathcal B}$.
Hence, if $x\in D_i\subset M_i$ then the image of $x$ in $\overline M$
is some $\overline v_\chi$. It follows that $h_\chi x=uv_\chi$, for some
unit $u\in R$, so in fact $v_\chi\in M_i$. If we replace $x$ by $v_\chi$,
in $B$, the resulting set is again a left SNF basis for $\beta$. This means that
the set of all the $v_\chi$ is a left SNF basis for $\beta$. An identical argument
about the construction of a right SNF basis from the $N_j$ shows that
the set of all $v_\chi$ is also a right SNF basis for $\alpha$.
\end{proof}

A  characterization of the SNF by some simple properties is given in \cite{Queiro}.
Other general properties of the Smith group and some applications are discussed in 
  \cite{rushanan2} and \cite{rushanan1}.

\section{The SNF problem for incidence matrices} 
In applying the general theory of the previous section to incidence problems 
we start with $M=\Z^X$ and $N=\Z^Y$ and
$\eta:M\to N$, sending $x\in X$ to the sum of all elements of $Y$
incident with $x$. If there is an action of a group $G$ on $X$ and $Y$ that 
preserves the incidence relation, then $M$ and $N$ are
permutation $\Z G$-modules and $\eta$ is a $\Z G$-module homomorphism. 
For example in Example~\ref{exsets}, the symmetric group $S_n$
acts, and in Example~\ref{exsubspaces} we may take $G=\GL(n,q)$.
The spaces $\Q^X$ and $\Q^Y$ have inner products for which $X$ and $Y$
are orthonormal bases, and then $M$ and $N$ are unimodular lattices.
If $\eta^*:N\to M$ is the transpose of $\eta$, sending $y\in Y$
to the sum of elements of $X$ incident with it, then we have
$$
\langle x,\eta^*(y)\rangle=\langle \eta(x),y\rangle.
$$
If we work over a discrete valuation ring $R$, then the
modules $M_i$ and $N_i$ of \S\ref{localsnf}  are $RG$-submodules
of $M$ and $N$. Further,  $\overline M$, $\overline M_i$, $\overline N$, $\overline N_i$
are $\overline RG$-modules. In some cases, one can find a very direct 
connection between the $RG$-submodule structure of these modules and the
SNF, using (\ref{eldiv}).

Many families of incidence relations are encompassed
by Examples~\ref{exsets} and \ref{exsubspaces}, 
and most of the SNF problems are unsolved. 
The complete solutions in some important cases and progress
in other cases will form a large part of our discussion.
A particular family may be an example of  a general combinatorial structure
such as a strongly regular graph, design or difference set, so
our next task is to  examine  properties of these general structures
that are relevant to the SNF problem.


\subsection{Graphs}
In the case where the incidence relation is between sets $X$ and $Y$
which are equal, or are in bijection is some prescribed way, the incidence
matrix $A$ can be regarded as the adjacency matrix of a directed graph, which is simple
if no element is related to itself. 

In addition to the SNF one can also consider the eigenvalues of $A$.
This applies to some subcases of Examples~\ref{exsets} and \ref{exsubspaces},
such as when $X=Y$ is the set of $k$-subsets of an $n$-set, with
disjointness as the incidence, which defines the {\it Kneser graph} $K(n,k)$,
or when $X=Y$ is the set of lines in $PG(3,q)$ with the relation of
skewness. The latter example, to which we shall return later,
also defines the noncollinearity graph of a hyperbolic quadric in $PG(5,q)$,
by the Klein correspondence. 
At the most general level, there is no close relationship between the eigenvalues and 
invariant factors of a square integral matrix, as illustrated by the  matrices 
\begin{equation}
A_1=\begin{pmatrix}2&0&0\\
0&2&0\\
0&0&2
\end{pmatrix}, \qquad
A_2=\begin{pmatrix}2&1&0\\
0&2&0\\
0&0&2
\end{pmatrix},\qquad
A_3=\begin{pmatrix}2&1&0\\
0&2&1\\
0&0&2
\end{pmatrix}
\end{equation}
whose invariant factors are respectively $(2,2,2)$, $(1,2,4)$ and $(1,1,8)$.

One general observation is the following \cite{BH}.
\begin{lemma}\label{eiginv}
Let $A$ be a square integral matrix with integral eigenvalue
$a$ of (geometric) multiplicity $m$. Then the number of invariant factors of $A$
divisible by $a$ is at least $m$.
\end{lemma}

Let $\Gamma$ be a regular graph of degree $k$ on $v$ vertices.
Then $\Gamma$ is a {\it strongly regular graph} (SRG) if, for any pair of vertices,
the number of common neighbors depends only on
whether the vertices are adjacent or not. If the number of
common neighbors for pairs of adjacent vertices is  $\lambda$
and the number for pairs of nonadjacent vertices is $\mu$, then
we say that $\Gamma$ is an SRG with parameters $(v,k,\lambda,\mu)$.
We refer to \cite{BH} for properties and many examples of SRGs.
An SRG has of course the all-one vector $\allone$ as an eigenvector
(with eigenvalue $k$). The orthogonal complement of $\allone$
(in the $\Q$-inner product space with the vertices as orthonormal basis) 
is the direct sum of eigenspaces for two distinct integral eigenvalues $r$ and $s$,
called the restricted eigenvalues.

The following \cite[p.174]{BH} gives quite strong information about the $p$-elementary
divisors of an SRG for $p\nmid v$.

\begin{proposition}\label{SRGeldiv}  Let  $A$ be the adjacency matrix of an SRG
with parameters $(v,k,\lambda,\mu)$ and restricted eigenvalues $r$ and $s$.
Let  $p$ be a prime and assume that $p\nmid v$, $p^a||k$, $p^b||s$, $p^c||r$, with
$a\geq b+c$. Let $e_i$ denote the multiplicity of $p^i$ as an elementary divisor of $A$.
Then $e_i=0$ for $\min(b,c)<i<\max(b,c)$ and $b+c<i<a$
and $i>a$. Moreover, $e_{b+c-i}=e_i$ for $0\leq i<\min(b,c)$. 
\end{proposition}

Of the two examples mentioned above, the noncollinearity graph of the Klein
quadric is always an SRG while the Kneser graphs never are.

\subsection{Abelian Cayley graphs and difference sets}
Suppose a (multiplicative) abelian group $G$ acts regularly on $X$, preserving 
an incidence relation $I\subseteq X\times X$.
Then by identifying $G$ with $X$, we have a translation-invariant relation
on $G$. Such relations are uniquely determined by the
subset $E=\{e\in G\mid (1, e)\in I$\}, since $(g,h)\in I$ if and only if 
$g^{-1}h\in E$. The relation can be viewed as adjacency in the Cayley graph of $G$
with respect to the connecting set $E$. Let $A$ be the adjacency matrix with respect to 
some fixed order.

Let $C=(\chi(g))$ denote the character table of $G$, with rows 
indexed by the set $G^\vee$ of irreducible complex characters of $G$ and the columns
indexed by the elements of $G$, in the same order as for $A$ and 
$\overline C=(\chi(g^{-1}))$.
 
Then, as first observed in \cite{McWM}, we have
\begin{equation}\label{McWM}
\frac{1}{\abs{G}}{\overline C}AC^t=\diag(\chi(E))_{\chi\in G^\vee}.
\end{equation}

The significance of this equation is twofold.
First, the orthogonality relations $\frac{1}{\abs{G}}\overline C C^t=I$ show that
the $\chi(E)$ are the eigenvalues of $A$.
Secondly, if $p$ is any prime that does not divide $\abs{G}$, we can read
the equation as an equivalence over the ring $R=\Z_p[\xi]$,
where $\Z_p$ is the ring of $p$-adic integers and $\xi$ is a primitive
$\abs{G}$-th root of unity. Since $p$ is unramified in $R$,
we see that the exact powers of $p$ dividing the $\chi(E)$ in $R$ are
precisely the $p$-elementary divisors of $A$.

Here, the general theory ends and the known methods for computing these valuations
depend very much on the prime and the relation in question. 

\begin{example}\label{hammingmax}
Let $S$ be a set of size $mn$ partitioned into $m$ subsets of
size $n$. Let $X$ be the set of {\it transversals}, that is, subsets of $S$ of size $m$
that contain one element from each part. Two transversals are incident if and
only if they are disjoint. 
Thus the incidence matrix $A$ is $n^m\times n^m$.
Let $Z_n$ be a cyclic group of order $n$ written multiplicatively.
We can identify $X$ with the elements of the group $G=(Z_n)^m$ and observe that the
regular action of $G$ on itself by multiplication preserves incidence.
The set of elements incident with $1_G$ is $E=\{(a_1,\ldots,a_m)\in G\mid a_i\neq 1
\, \forall i\}$.  Let $A$ be the incidence matrix, for some fixed ordering of $G$.
We will apply the equation (\ref{McWM}) first to compute the spectrum of $A$ 
and then to obtain the SNF.
First  we see from  (\ref{McWM}) that $A$ is similar to $\diag(\chi(E))_\chi$.
The irreducible characters of $G$ are of the form 
$\chi=(\lambda_1,\ldots,\lambda_m)$, where
$\lambda_i$ is an irreducible character of $Z_n$. Then starting from the fact that 
$$
\sum_{z\in Z_n\setminus\{1\}}\lambda_i(z)=\begin{cases} -1,\quad\text{if $\lambda_i$
is not principal,}\\
n-1,\quad\text{if $\lambda_i$ is principal},
\end{cases}
$$
it follows that $\chi(E)=(-1)^{m-r}(n-1)^r$, where $r$ is the number
of $i$ such that $\lambda_i$ is principal. Also, one sees that the multiplicity
of  $(-1)^{m-r}(n-1)^r$ as an eigenvalue is the number of characters $\chi$
that have exactly $r$ principal components, which is $\tbinom{m}{r}(n-1)^{m-r}$,
since once the $r$ principal components are fixed, there are $n-1$ choices for
nonprincipal characters in each of the remaining ${m-r}$ components.
Thus we know that the determinant of $A$ is, up to sign, a power of $(n-1)$ and
the SNF involves only primes $p$ dividing $n-1$. If $p$ is such a prime,
then in particular $p\nmid \abs G$ and we can view (\ref{McWM}) as expressing
equivalence of matrices over a suitable $p$-local ring.
We may conclude that in a suitable ordering of the characters, $\diag(\chi(E))_\chi$ is actually the SNF of $A$. 
The matrix $A$ in this example is the association matrix for the
maximal distance in the {\it Hamming association scheme} $H(m,n)$.
(For background on association schemes, see \cite{BrouwerHaemers}.)
The eigenvalues of all association matrices for $H(m,n)$ are known \cite{Delsarte},
but the SNFs of these matrices do not seem to be known. This example was studied in \cite{JMU}, where the  SNF was computed for small examples and conjectured in general.

\end{example}

Cayley Graphs can be derived from {\it difference sets}.
An {\it abelian difference set} is a subset $B$ in an abelian group 
$G$ such that for some natural number $\lambda$ each nontrivial element $g\in G$
has the form $g=x^{-1}y$ for precisely $\lambda$ pairs $(x,y)$ of elements in $B$.
If we take $G$ as the set of points and the translates
$gB$ as the set of blocks, the incidence structure obtained 
has the structure of a {\it symmetric design} (cf. \cite{Lander1}).
Since $gB=g'B$ only for $g=g'$, we can also identify the set of blocks 
with $G$, and then we have a $G$-invariant relation on $G$ of the kind
described above, hence a Cayley graph, with $E=B$.

\begin{example}\label{singerdiff}
Let $d\geq 3$ and let $q$ be a power of a prime $p$.
The Singer difference set is a difference set $L_0$ in $\F^*_{q^d}/\F^*_q$, 
consisting of those cosets of $\F^*_q$ whose elements $y$ satisfy $\Tr_{\F_{q^d}/\F_q}(y)=0$. Each character $\chi$ of $\F^*_{q^d}/\F^*_q$ is of the form 
$\omega^{d(q-1)}$, where $\omega$ is a generator of the character group of $\F^*_{q^d}$. 
The {\it Gauss sum} over $\F_{q^d}$ with respect to the multiplicative character
$\chi$ and the additive character $y\mapsto \xi^{\Tr_{q^d/p}(y)}$, where $\xi$
is a primitive $p$-th root of unity is defined as 
$$
g(\chi)=\sum_{y\in\F^*_{q^d}}\chi(y)\xi^{\Tr_{q^d/p}(y)}. 
$$
Evaluation of this Gauss sum (\cite{Yamamoto}; see also \cite[p.400]{Berndt-Evans-Williams}) yields
$$
g(\chi)=q\chi(L_0).
$$
The $\mathfrak p$-adic valuation of this Gauss sum is determined by a classical theorem
of Stickelberger. Computation of the SNF is then a matter
of determining the $\mathfrak p$ valuation of $\omega^d$ in terms of $d$
and counting the $d$'s for a given valuation. In this way,
a new proof of the  $p$-rank was given in \cite{EHKX}.  
R. Liebler also took this approach towards computing the SNF 
in unpublished work.

The incidence relation for Singer difference sets can also be
interpreted as the incidence of points and hyperplanes in $PG(n,q)$ , $n\geq 2$,
and so it is both a special case of Example~\ref{exsubspaces} and a generalization of
the point-line incidence of the projective planes $PG(2,q)$, for which
the SNF was computed in \cite{Lander2}.
This point-hyperplane incidence for general $n$ but prime fields only was
studied by Black and List, who determined the SNF in \cite{BL}. 
Their method was different from the
general Cayley graph approach outlined above; instead they viewed the 
incidence matrix as the so-called rational character table of an elementary
abelian group and made a reduction by  tensor factorization to the case of a cyclic group of prime order. For the case of arbitrary finite fields, the SNF
of the point-hyperplane incidence was first computed using 
the modular representation theory of $\GL(n+1,q)$. 

Let $q=p^t$. Let $d_\lambda$ be the coefficient 
of $z^{\lambda}$ in $(\sum_{j=0}^{p-1}z^j)^{n+1}$.
Explicitly,
\begin{equation}\label{dlambda}
d_{\lambda} = \sum_{k = 0}^{\lfloor \lambda/p \rfloor}(-1)^{k}\binom{n+1}{k}\binom{n+\lambda-kp}{n}.
\end{equation}

Define the matrix $M=(m_{i,j})$ $(1\le i, j\le n)$ with 
entries in $\Z[z]$
by $m_{i,j}=d_{pj-i}z^{pj-i}$ and let $a_r=a_{r,t}$ $(0\le r\le t(n-1))$ be the coefficient of
$z^{r(p-1)}$ in $\trace(M^t)$.

\begin{theorem}\label{singer}
The Smith group of the incidence matrix of points and hyperplanes of $PG(n,q)$ 
has cyclic factors of the following orders and multiplicities:
\begin{enumerate}
\item $\frac{(q^n-1)}{(q-1)}$ with multiplicity 1.
\item $p^{r}$ with multiplicity $a_{r}$, $0\le r\le (n-1)t$.
\end{enumerate}
\end{theorem}

It is worth mentioning some  representation-theoretic perspectives
of this result. The group $G=\GL(n+1,q)$ acts on $X=PG(n,q)$ and so
$\F_q^X$ is a module over the group algebra $\F_qG$.
The {\it socle} of a module $E$, denoted  $\soc(E)$, 
is the sum of all simple submodules or, equivalently, the maximal
semisimple submodule. The {\it radical}, $\rad(E)$,
is the intersection of all maximal submodules or, equivalently,
the smallest submodule by which the quotient
module is semisimple. The higher radicals and socles
are defined recursively in the usual way: $\soc^i(E)$
is the full preimage in $E$ of $\soc(E/\soc^{i-1}(E))$ and
$\rad^i(E)=\rad(\rad^{i-1}(E))$.
Let $Y$ denote the set of hyperplanes in $PG(n,q)$ and let
$R=\Z_p[\omega]$ be the extension of the $p$-adic
integers by a primitive $(q-1)$-th root of unity, so that $R/pR\cong\F_q$.
Consider the incidence map
$\eta:R^Y\to R^X$, and let the modules $M_i\subset R^Y $ and $N_i\subset R^X$
be defined as in \S\ref{localsnf}, with corresponding submodules $\overline M_i\subset \F_q^Y$ and $\overline N_i\subset \F_q^X$. 
Then it can be shown that $\overline M_i=\rad^i \F_q^Y$ and 
$\overline N_i=\soc^i \F_q^X$. The radical and socle series of $\F_q^X$ are described
in \cite{BS}.
According to Theorem~\ref{singer} the trace of the matrix $M^t$ 
is the generating function for the multiplicities of the $p$-elementary divisors.
The multiplicity of $p^r$ is the coefficient $a_r$ of $z^{r(p-1)}$ in  this polynomial,
which  is a sum of $t$-fold products $\prod_{i=1}^td_{j_i}$ where $\sum_{i=0}^tp^id_{j_i}=r(p-1)$. As explained in \cite{BS} these $t$-fold products are the dimensions of simple $\F_q\GL(n+1,q)$-modules, factorized as $t$-fold twisted tensor products in accordance with Steinberg's tensor product theorem (\cite[II/3.17]{Jantzen}).

Theorem~\ref{singer} is a special case of a more general result, to be treated 
in \S\ref{subspaces}, from \cite{CSX}, which solves the SNF problem for points versus subspaces of a fixed dimension in $PG(n,q)$. 
\end{example}

The cyclic difference sets studied in \cite{CX} also arise from
multiplicative groups of finite fields,  the sums $\chi(B)$ are evaluated using 
Jacobi sums. Here too, Stickelberger's Theorem applies because of the simple relation between Jacobi and Gauss sums.

Aside from adjacency matrices of graphs, there is also the vertex-edge incidence
matrix. A recent paper in this direction is \cite{Wong1}, in which
the SNF problem for the vertex-edge incidence matrices of a certain class of 
bipartite graphs is tackled. The approach makes use of Smith's 
characterization of the invariant factors in terms of the determinantal 
divisors $d_i$ mentioned in the Introduction.
One application of the results is a new calculation of zero-sum mod 2 bipartite
Ramsey numbers.

We now turn to a more detailed review of the current state
of knowlege about Examples~\ref{exsets} and \ref{exsubspaces}.
 
\subsection{Subsets of a set}\label{subsets}
Let $S$ be a finite set of size $n$ and let $X_k$ denote the 
set of subsets of $S$ of size $k$. 
Let $W_{t,k}$ be the inclusion matrix of $t$-subsets in $k$-subsets.
We think of $W_{t,k}$ as a map $\Z^{X_t}\to\Z^{X_k}$.

A diagonal form was found in \cite{Wilson}.
\begin{theorem}\label{Wilsonform}
$W_{t,k}$ has a diagonal form with diagonal entries $\tbinom{k-i}{t-i}$, each
with multiplicity $\tbinom{n}{i}-\tbinom{n}{i-1}$, for $0\leq i\leq t$.
\end{theorem}

Indispensable ingredients in the proof of these results
are the following  fundamental identities, which are easily verified.
\begin{equation}
\begin{aligned}
W_{i,j}W_{j,t}=\binom{t-i}{j-i}W_{i,t},\qquad & W_{i,j}\overline W_{j,t}=\binom{n-t-i}{j-i}\overline W_{i,t}\\
W_{t,k}=\sum_{i=0}^t(-1)^iW_{i,t}^T\overline W_{i,k}, \qquad& 
\overline W_{t,k}=\sum_{i=0}^t(-1)^iW_{i,t}^TW_{i,k}.
\end{aligned}
\end{equation}
Here, $\overline W_{t,k}$ is the disjointness matrix of $t$-subsets and $k$-subsets.
The matrices $W_{t,k}$ are of course the same as the matrices $\overline W_{t,n-k}$ 
with the columns reordered, since  a $t$-subset is contained in a $k$-subset
if and only if it is disjoint from the complement. Thus, Wilson's formula
also solves the SNF problem for the disjointness relation.

The symmetric group $S_n$ acts on $S$ and $W_{t,k}$
is a $\Z S_n$-module map, so although this action is not used in
the original proof, it seems appropriate nevertheless
to mention some connections of Theorem~\ref{Wilsonform} with the representation theory of
$S_n$. The multiplicities are the dimensions of the simple 
$\Q S_n$-submodules of $\Q^{X_t}$ and, over arbitrary fields, of {\it Specht modules} 
corresponding to partitions with two parts \cite {JamesKerber} . 
The diagonal entries yield the $p$-rank of $W_{t,k}$
over $\F_p$, which is helpful in understanding the $\F_pS_n$-submodule structure of $\F_p^{X_k}$.  

In \cite{Bier}, it was shown that with respect to the {\it Frankl rank} \cite{Frankl}
of a subset, there is a canonical choice of $\tbinom{n}{i}-\tbinom{n}{i-1}$ rows 
of $W_{i,k}$, for $i\leq k$, such that the union is a basis of $\Z^{X_k}$
which is a right SNF basis for every $W_{t,k}$, $t\leq k$.

The general case of Example~\ref{exsets} is where $s$, $t$ and $k$ are fixed
and a $t$-subset is incident with a $k$-subset if and only if their
intersection has size $s$. 
In the case $t=k\leq n/2$, this defines the $(k-s)$-th association
relation in the {\it Johnson association scheme} $J(n,k)$. As for the Hamming schemes, the
eigenvalues of the association matrices for $J(n,k)$ have been calculated long ago
(by Yamamoto et. al. \cite{Yamamoto2} and independently by Ogasawara). However the SNF
problems remain open. A recent paper \cite{WilsonWong} gives a diagonal form for matrices
in the Bose-Mesner Algebra of $J(n,k)$ that satisfy a ``primitivity''
condition. However, this  condition is not satisfied by the association matrices themselves.

\begin{problem}
Solve the SNF problems for the association matrices of the Hamming and Johnson schemes.
The answer is known for the relation of maximal distance in both cases.
For the Hamming scheme we saw in Example~\ref{hammingmax}
that the diagonal entries of a diagonal form were equal, counting multiplicities, to the eigenvalues. This is also true of $J(n,k)$, where the maximal distance relation is the disjointness relation that defines the Kneser graphs $K(n,k)$.  
The eigenvalues of $K(n,k)$ are the same (up to sign) as the entries of the diagonal form of Theorem~\ref{Wilsonform}.  
\end{problem}

\subsection{Subspaces of a vector space}\label{subspaces}
We already considered the incidence of points and hyperplanes of $PG(n,q)$
in Example~\ref{singerdiff}.
Now, we shall look at some further instances of Example~\ref{exsubspaces}.
We shall assume that $V$ is a vector space of dimension $n+1$ over the field
$\F_q$, with $q=p^t$. For $0\leq d\leq n+1$ we denote by $\LL_d$ the set of $d$-dimensional subspaces of $V$ (``$d$-subspaces'' for short). By analogy with the subsets of a set,
we consider for $d\leq e$ the inclusion matrices $A_{d,e}$ of $d$-subspaces
in $e$-subspaces and the incidence matrices $\Abar_{d,e}$ for the relation of zero
intersection, whenever $d+e\leq n+1$.

The most general result obtained so far is for the matrices $A_{1,r}$,
where the SNF problem was solved in \cite{CSX}. As the two statements below
show, the Smith group is a product of a cyclic group of order coprime to $p$
(cyclic $p'$-group for short) and a large $p$-group, the determination of which is the main work. The statement of this and later results involves a certain partially ordered
set.  Let $\HH$ denote the set of $t$-tuples of integers $\mathbf{s}=(s_{0}, \dots , s_{t-1})$ that satisfy, for $0\le i\le t-1$,
\begin{enumerate}
\item $1 \le s_{i} \le n$,
\item $0 \le ps_{i + 1} - s_{i} \le (p-1)(n+1)$,
\end{enumerate}
with subscripts read modulo $t$.  First introduced in~\cite{Hamada},
 the set $\HH$ was later used in~\cite{BS} to describe the module structure of $\F_q^{\LL_1}$ under the action of $\GL(n+1,q)$.  
Let
\begin{equation}\label{Halpha}
\HH_{\alpha}(s) = \Bigl\{(s_{0}, \dots , s_{t-1}) \in \HH \, \Big\vert \, \sum_{i=0}^{t-1}\max\{0, s-s_{i}\} = \alpha\Bigr\}.
\end{equation}

To each tuple $\mathbf{s} \in \HH$ we associate a number $d(\mathbf{s})$ as follows.  For $\mathbf{s} = (s_0, \dots ,s_{t-1}) \in \HH$ define the integer tuple $\mathbf{\lambda} = (\lambda_{0}, \dots ,\lambda_{t-1})$ by 
\[
\lambda_{i} = ps_{i+1} - s_{i} \hbox{\quad (subscripts mod $t$)}.
\]
Finally, set $d(\mathbf{s}) = \prod_{i=0}^{t-1}d_{\lambda_{i}}$, where
$d_\lambda$ is  defined in (\ref{dlambda}).

\begin{theorem}\label{p'part} Let $v=\abs{\LL_1}$.
 The invariant factors of $A_{1,r}$ are all 
$p$-powers except for the $v^\mathrm{th}$ invariant,
which is a $p$-power times $(q^r-1)/(q-1)$.
\end{theorem}

\begin{theorem}\label{main}
The $p$-adic invariant factors of the incidence matrix $A_{1,r}$ between
$\mathcal{L}_1$ and $\mathcal{L}_r$ are $p^\alpha$,
$0\leq \alpha\leq (r-1)t$, with multiplicity
$$e_\alpha=\sum_{\mathbf{s}\in\HH_\alpha}
d(\mathbf{s})+ \delta(0,\alpha)$$ 
where 
\begin{equation}
\delta(0, \alpha)=
\left\{
\begin{array}{ll}
1, & \mbox{if} \; \alpha=0,\\
0, & \mbox{otherwise}. \end{array}
\right.
\end{equation}
\end{theorem}

Several themes we have discussed are reflected in the proof of Theorem~\ref{main}. 
The modules $M_i$ of \S\ref{localsnf} and 
their relation to the $\F_q\GL(n+1,q)$-submodule
structure of $\F_q^{\LL_1}$ play an important role. The space
has a basis of monomials and a left SNF basis is constructed by
taking Teichm\"uller lifts of these monomials to a suitable
$p$-adic ring $R$. Gauss sums and Stickelberger's Theorem then
appear, but in a rather different way from equation (\ref{McWM}), 
in computations  in  the $p$-adic group ring. 
A key ingredient is a result of Wan \cite{Wan}, from which
one obtains the correct upper bounds for the $p$-elementary divisors.

In the case of $A_{1,r}$, the $p'$-part of the Smith group is cyclic.
In his Ph.D. thesis \cite{Chandler}  Chandler, using results of James \cite{James}, 
 has given a diagonal form over the $\ell$-adic
integers for $\ell\neq p$ for all of the $A_{d,e}$. The result bears a
remarkable resemblance to Wilson's diagonal form for subsets, with binomial
coefficients replaced by $q$-binomial coefficients.

\begin{theorem} Let $s\leq r$ and  $s+r\leq n+1$. Let $\ell$ be any prime
not dividing $q$ and let $\Z_\ell$ denote the $\ell$-adic integers.
Then over $\Z_\ell$ the matrix $A_{r,s}$ has a diagonal form whose
diagonal entries are
$\tbinom{r-i}{s-i}_q$ with multiplicity $\tbinom{n+1}{i}_q-\tbinom{n+1}{i-1}_q$.
\end{theorem}

Note that whereas, in the case of subsets, $\overline W_{t,k}$ is essentially the same
as $W_{t, n-k}$ using set complementation, there is no simple relation
between $A_{d,e}$ with any $\Abar_{r,s}$, except in the case $d=1$,
where $A_{1,e}$ and $\Abar_{1,e}$ are complementary.  
In the case where $r+s=n+1$, the zero-intersection relation
encoded in $\Abar_{r,s}$ is an example of an {\it oppositeness}
relation in a spherical Tits building. In general, we know from \cite{Brouwer} that
for such relations all invariant factors are powers of the natural
characteristic $p$. The first nontrivial example is
when $\dim V=4$ and we consider the relation of zero intersection of $2$-dimensional
subspaces. Geometrically, we may think of skew lines in projective space. 
The SNF was determined in \cite{BDS}. 
Let $A=\Abar_{2,2}$.
\begin{theorem}
\label{thmBDS1}
Let $e_{i} = e_{i}(A)$ denote the multiplicity of $p^i$ as an elementary divisor of $A$.
\begin{enumerate}
\item \label{item:A1} $e_{i} = e_{3t-i}$ for $0 \le i < t$.
\item \label{item:A2} $e_{i} = 0$ for $t < i < 2t$, $3t < i < 4t$, and $i > 4t$.
\item \label{item:A3} $\sum_{i=0}^{t}e_{i} = q^4 + q^2$.
\item \label{item:A4} $\sum_{i=2t}^{3t}e_{i} = q^3 + q^2 + q$.
\item \label{item:A5} $e_{4t} = 1$.
\end{enumerate}
\end{theorem}
Thus we get all the elementary divisor multiplicities once we know $t$ of the numbers $e_{0}, \dots , e_{t}$ (or the numbers $e_{2t}, \dots ,e_{3t}$).  The next theorem describes these.  To state the theorem, we need some notation.

Set
\[
[3]^{t} = \{(s_0, \dots ,s_{t-1}) \, | \, s_{i} \in \{1,2,3\} \hbox{ for all $i$}\}
\]
and
\[
\mathcal{H}(i) = \bigl\{(s_0, \dots ,s_{t-1}) \in [3]^{t} \, \big\vert \, \#\{j|s_{j}=2\}=i\bigr\}.
\]
In other words, $\mathcal{H}(i)$ consists of the tuples in $[3]^{t}$ with exactly $i$ twos.  To each tuple $\mathbf{s} \in [3]^{t}$ we associate a number $d(\mathbf{s})$ as follows.  For $\mathbf{s} = (s_0, \dots ,s_{t-1}) \in [3]^{t}$ define the integer tuple $\mathbf{\lambda} = (\lambda_{0}, \dots ,\lambda_{t-1})$ by 
\[
\lambda_{i} = ps_{i+1} - s_{i},
\]
with the subscripts read mod $t$. Since $n+1=4$, the integer $d_{k}$ defined 
previously in (\ref{dlambda}) is the coefficient of $x^{k}$ in the expansion of $(1 + x + \cdots + x^{p-1})^{4}$.  Also recall that $d(\mathbf{s}) = \prod_{i=0}^{t-1}d_{\lambda_{i}}$.
\begin{theorem}
\label{thmBDS2}
Let $e_{i} = e_{i}(A)$ denote the multiplicity of $p^i$ as an elementary divisor of $A$.  Then, for $0 \le i \le t$,
\[
e_{2t+i} = \sum_{\mathbf{s} \in \mathcal{H}(i)}d(\mathbf{s}).
\]
\end{theorem}

\begin{remark}
When $p=2$, notice that $d(\mathbf{s})=0$ for any tuple $\mathbf{s}$ containing an adjacent $1$ and $3$ (coordinates read circularly).  Thus the sum in Theorem~\ref{thmBDS2} is significantly easier to compute in this case. 
\end{remark}

The proofs involve a combination
of the methods already mentioned. The first observation is that
the skewness relation defines a strongly regular graph which has integral eigenvalues 
which are powers of $p$ up to signs. Thus, we may apply Proposition~\ref{SRGeldiv},
which leads to Theorem~\ref{thmBDS1}.
The proof of Theorem~\ref{thmBDS2} lies somewhat deeper.
The missing elementary divisors can be shown to be the same
as for the composite $\Abar_{2,1}\Abar_{1,2}$.
The elementary divisors of $\Abar_{2,1}$ and $\Abar_{1,2}$ are known from \cite{CSX}.
Then to calculate the SNF of $\Abar_{2,1}\Abar_{1,2}$ use must be made of  the 
multiplicity-free action of the Singer cycle on $k^{\LL_1}$, through Proposition~\ref{multfree}.  
 
In fact the following more general theorem on composite incidence maps
for $V$ of arbitrary dimension is proved in \cite{BDS}, for arbitrary $r$ and $s$.
Let $\HH_{\beta}(r)$ be as defined in (\ref{Halpha}) and 
$$
\begin{aligned}
\leftsub{\beta}{\HH}(r) &= \{(n+1-s_{0}, \dots , n+1-s_{t-1}) \, | \, (s_{0}, \dots , s_{t-1}) \in \HH_{\beta}(r)\} \\
&= \Bigl\{(s_{0}, \dots , s_{t-1}) \in \HH \, \Big\vert \, \sum_{i=0}^{t-1} \max\{0, s_{i}-(n+1-r)\} = \beta\}\Bigr\}.
\end{aligned}
$$

\begin{theorem}\label{gencompthem}
Let $e_{i}=e_{i}(\Abar_{r,1}\Abar_{1,s})$ denote the multiplicity of $p^{i}$ as a $p$-adic elementary divisor of $\Abar_{r,1}\Abar_{1,s}$.%
%
%
\begin{enumerate}
\item $e_{t(r+s)} = 1$.
\item For $i \neq t(r+s)$,
\[
e_{i} = \sum_{\mathbf{s} \in \Gamma(i)} d(\mathbf{s}),
\]
where
\[
\Gamma(i) = \bigcup_{\substack{\alpha + \beta = i \\ 0 \le \alpha \le t(s-1) \\ 0 \le \beta \le t(r-1)}}  \leftsub{\beta}{\HH}(r) \cap \HH_{\alpha}(s).
\]
\end{enumerate}
Summation over an empty set is interpreted to result in $0$.
\end{theorem}

\begin{problem} Opposite Subspaces Problem. Let $r+s=n+1$ and let $V$
be an $(n+1)$-dimensional vector space over $\F_q$.  Solve the SNF problem
for the incidence between the set of $r$-dimensional subspaces and $s$-dimensional subspaces, where incidence is defined as zero intersection.  
\end{problem}

\subsection{Spaces with forms}
Every question about incidence of subspaces in a vector space
prompts analogous questions about vector spaces with bilinear,
quadratic or hermitian forms. In these cases, instead of all subspaces
one considers a distinguished class, such as totally singular
subspaces and their perp spaces with respect to the form. 
In the case of symplectic forms, Lataille \cite{Lataille}
has solved the SNF problem for symplectic vector spaces over 
a prime field $\F_p$, as well as computing the $p'$- factor of the
Smith module for the general case over any finite field.
In the case of a $6$-dimensional vector space with a quadratic form
of maximal index, the points of the quadric are the lines of $PG(3,q)$
under the Klein correspondence, so the result of \cite{BDS} discussed
above solves the SNF problem for the relation of (non)collinearity.
If we have a $3$-dimensional vector space over $\F_{q^2}$
with a nonsingular Hermitian form, then we may consider the incidence
of points and lines of the Hermitian unital. The SNF problem for this case was
solved in \cite{Hiss}, using results in modular representation theory.
In many other cases, there are partial answers, such as the computation
of the $p$-rank, which equals the multiplicity of 1 as a $p$-elementary divisor.
Very little else is known, so it seems reasonable to start in low dimensions.
Examples of incidence structures based on low-dimensional vector spaces with forms include the {\it generalized quadrangles}. (See \cite{PayneThas}.) These include the 
classical point line geometries of singular points and totally singular lines 
in symplectic spaces of (vector) dimension 4, orthogonal spaces of dimension 5 and 6 and
in Hermitian spaces of dimensions 4 and 5. One can pose the SNF problem for 
the point-line incidence or for the collinearity relation.

\begin{problem} Solve the SNF problem, with respect to one of the
incidence relations, for a family of generalized quadrangles. 
There are no cases for which this problem has yet been solved.
\end{problem}

\subsection{A word about $p$-ranks}
For many of the incidence relations we have considered, the $p$-rank,
i.e. the rank of the incidence matrix considered over $\F_p$, has been found,
which is the same as computing the multiplicity of $1$ as a $p$-elementary
divisor. In the case of subsets, the problem of determining
the $p$-ranks of the $i$-association matrices of the Johnson scheme $J(n,k)$
for all $p$ would appear to be practically as hard as the full SNF problem.
In the case of subspaces the $p$-ranks of $A_{d,e}$ are still unknown, but
the $p$-ranks of $\overline A_{d,e}$ are given by \cite{Sin2}. 
The $p$-ranks of some of the generalized quadrangles are known when $p$
is the defining characteristic, and all of the cross-characteristic ranks are known,
(from \cite{Sin-Tiep} and the references cited there). Also, many of the point-hyperplane $p$-ranks (in the defining characteristic) can be found in \cite{ArslanSin}. These are examples of oppositeness relations, or complements of such. 
For oppositeness relations in the building of a finite group of Lie type of characteristic
$p$ it is known  (\cite{Sin3}) that the $p$-ranks  are dimensions of 
irreducible $p$-modular representations of the group. 

The general topic of $p$-ranks has been studied for a very wide
variety of incidence matrices, especially for $p=2$. One reason for this is that
an incidence matrix can be used  as a parity-check or generator matrix of a 
binary code, whose dimension is then given by the $2$-rank. 
The literature is too extensive to summarize here. In the case of designs, many references, examples, applications and open questions are described in \cite{Xiang2}. Some recent papers on $2$-ranks of incidence structures of certain points and lines in $PG(2,q)$ defined by a conic are \cite{SWX}, \cite{Wu1} and \cite{Wu2}.

\section*{Acknowledgements}
I would like to thank Qing Xiang and Josh Ducey for helpful
discussions during the preparation of this article. 
\bibliographystyle{amsplain}
\def\cprime{$'$}
\providecommand{\bysame}{\leavevmode\hbox to3em{\hrulefill}\thinspace}
\providecommand{\MR}{\relax\ifhmode\unskip\space\fi MR }
\providecommand{\MRhref}[2]{%
  \href{http://www.ams.org/mathscinet-getitem?mr=#1}{#2}
}
\providecommand{\href}[2]{#2}

\end{document}